\pdfoutput=1
\ifx\thesis\undefined
\documentclass[12pt,twoside,a4paper]{amsart}
\usepackage{amssymb}
\usepackage{amsmath,amscd}
\usepackage{epic}
\usepackage{bm}
\usepackage[all]{xy}
\usepackage[initials,nobysame,shortalphabetic,lite]{amsrefs}

\newtheorem{thm}{Theorem}[section]
\newtheorem{lma}[thm]{Lemma}

\theoremstyle{definition}

\theoremstyle{remark}

\newtheorem{preremark}{Remark}
\newtheorem{preex}{Example}

\numberwithin{equation}{section}

\begin{document}

\title{Tight maps and holomorphicity, exceptional spaces}


\author{Oskar Hamlet, Takayuki Okuda}

\address{Department of Mathematics\\Chalmers University of Technology and the University of Gothenburg\\412 96 G\"OTEBORG\\SWEDEN}

\email{hamlet@chalmers.se}

\address{Department of Mathematics, 
Graduate School of Science, Hiroshima University
1-3-1 Kagamiyama, Higashi-Hiroshima, HIROSHIMA 739-8526 JAPAN}

\email{okudatak@hiroshima-u.ac.jp}



\fi


\begin{abstract}
We show that there are no tight nonholomorphic maps from irreducible domains into exceptional codomains, the only exception being the already known tight nonholomorphic maps from the Poincar\'e disc. This follows up on previous work by the first author where this was shown for classical codomains.
\end{abstract}

\maketitle

\section{Introduction}
Let $(\mathcal{X}_i,\omega_i)$, $i=1,2$, be Hermitian symmetric spaces of noncompact type paired with some choice of invariant K\"ahler forms. A map $\rho\colon\mathcal{X}_1\rightarrow\mathcal{X}_2$ is called \emph{totally geodesic} if the image of every geodesic in $\mathcal{X}_1$ is a geodesic in $\mathcal{X}_2$. A totally geodesic map $\rho\colon \mathcal{X}_ 1\rightarrow \mathcal{X}_2$ satisfies 
\begin{equation}\label{t23}
sup_{\Delta\in\mathcal{X}_1}\int_\Delta{\rho^*\omega_2}\leq sup_{\Delta\in\mathcal{X}_2}\int_\Delta{\omega_2}
\end{equation}
 where the supremum is taken over triangles with geodesic sides. We say that the map is \emph{tight} if equality holds in (\ref{t23}).

The motivation for the study of tight maps comes from a structure theorem due to Burger-Iozzi-Wienhard, according to which the action on a Hermitian symmetric space $\mathcal{X}$ coming from a maximal representation preserves a tightly embedded subspace $\mathcal{Y}\subset\mathcal{X}$, 
\cite{C10}. 
It is important to classify these tight inclusions since this narrows the possible images of maximal representations.

Tight maps were first introduced by Burger-Iozzi-Wienhard \cite{C8}. In the paper they investigated the properties of tight maps and classified tight maps from the Poincar\'e disc. Among these they found both holomorphic and nonholomorphic tight maps. They were however unable to find nonholomorphic tight maps from a domain of larger dimension. They expressed their belief that no such nonholomorphic tight maps existed. In a previous paper the first author proved this for classical codomains and also for a part of the cases with exceptional codomains:

\begin{thm}[\cite{C32}]\label{oldmain}
Let $\mathcal{X}_ 1$ and $\mathcal{X}_2$ be irreducible Hermitian symmetric spaces, where we assume that $\mathcal{X}_ 1$ is not the Poincar\'e disc. Assume further that
\begin{enumerate}
\item $\mathcal{X}_ 2$ is classical, or
\item $\mathcal{X}_2$ is the exceptional Hermitian symmetric space associated to the symmetric pair $(\mathfrak{e}_{6(-14)},\mathfrak{so}(10)+\mathbb{R})$ and $\mathcal{X}_1$ is of rank at least two.
\end{enumerate}
If $\rho\colon \mathcal{X}_ 1\rightarrow \mathcal{X}_2$ is a tight map, then it is (anti-) holomorphic. 
\end{thm}
In this paper we finish the remaining exceptional cases, namely:

\begin{thm}\label{main3}
Let $\mathcal{X}_ 1$ and $\mathcal{X}_2$ be irreducible Hermitian symmetric spaces, where we assume that $\mathcal{X}_ 1$ is not the Poincar\'e disc. Assume further that
\begin{enumerate}
\item $\mathcal{X}_ 2$ is the exceptional Hermitian symmetric space associated to the symmetric pair $(\mathfrak{e}_{7(-25)},\mathfrak{e}_6+\mathbb{R})$, or
\item $\mathcal{X}_2$ is the exceptional Hermitian symmetric space associated to the symmetric pair $(\mathfrak{e}_{6(-14)},\mathfrak{so}(10)+\mathbb{R})$ and $\mathcal{X}_1$ is a Hermitian symmetric space of rank one.
\end{enumerate}
If $\rho\colon \mathcal{X}_ 1\rightarrow \mathcal{X}_2$ is a tight map, then it is (anti-) holomorphic. 
\end{thm}

Theorem \ref{oldmain} and \ref{main3} together yields that there are no nonholomorphic tight maps $\rho\colon\mathcal{X}_1\rightarrow\mathcal{X}_2$ when $\mathcal{X}_1$ is an irreducible Hermitian symmetric space not isometric to the Poincar\'e disc. This shows that the classification of tight holomorphic maps in \cite{C31} together with the classification of (nonholomorphic) tight maps from the Poincar\'e disc in \cite{C8} yields a complete classification of tight maps from irreducible Hermitian symmetric spaces.

\section{Preliminaries}
In the coming subsections we gather the tools we need. Before we do that, let us set some notation and mention a few facts about equivalence.

Let $(\mathcal{X}_i,\omega_i)$, $i=1,2$, be Hermitian symmetric spaces of noncompact type paired with some choice of invariant K\"ahler forms.
Let $G_1$ (resp. $G_2$) denote the groups of holomorphic isometries of $\mathcal{X}_1$  (resp. $\mathcal{X}_2$). We say that two totally geodesic maps $\rho,\eta\colon\mathcal{X}_1\rightarrow\mathcal{X}_2$ are \emph{equivalent} if $\rho=g\circ\eta$ for some $g\in G_2$. Holomorphic isometries preserve the K\"ahler form of $\mathcal{X}_2$. The notion of tightness is thus well defined for equivalence classes of maps.

Totally geodesic maps $\rho\colon\mathcal{X}_1\rightarrow\mathcal{X}_2$ correspond to Lie algebra homomorphisms $\rho\colon\mathfrak{g}_1\rightarrow\mathfrak{g}_2$ between the Lie algebras of $G_1$ and $G_2$ \cite[p.~225]{C3}.
The corresponding notion of equivalence is that two homomorphisms $\rho,\eta\colon\mathfrak{g}_1\rightarrow\mathfrak{g}_2$
are \emph{equivalent} if they only differ by an inner automorphism of $\mathfrak{g}_2$.

We will frequently work with complex Lie algebras and complexifications of our Hermitian Lie algebras and the homomorphisms between them. We will denote by $\mathfrak{g}_\mathbb{C}$ either a complex Lie algebra or the complexification of a real Lie algebra $\mathfrak{g}$. In the same manner we will denote complex homomorphisms and complexifications of real homomorphisms by a subscript $\mathbb{C}$. It should be clear from the context which is meant. Two complex Lie algebra homomorphisms $\rho_\mathbb{C}, \eta_\mathbb{C}\colon\mathfrak{g}_\mathbb{C}\rightarrow\mathfrak{g}'_\mathbb{C}$ are said to be \emph{equivalent} if they only differ by a (complex) inner automorphism of $\mathfrak{g}'_\mathbb{C}$. The complexification of two equivalent homomorphisms $\rho, \eta\colon\mathfrak{g}\rightarrow\mathfrak{g}'$ are thus equivalent complex homomorphisms. The converse however, is not true. We will have to keep this in mind as most of the theory of complex Lie algebras frequently uses this type of equivalence.

We also fix matrix models for two Hermitian Lie algebras that will appear in the calculations. Let $M_{p,q}(\mathbb{C})$ denote the space of complex $p$ by $q$ matrices, $M_{p}(\mathbb{C}):=M_{p,p}(\mathbb{C})$, and define
\begin{flalign*}
&\mathfrak{su}(p,q):=
&\Big\{ 
\left(\begin{array}{cc}
A&B\\
B^*&C
\end{array}\right) 
:\, &A\in M_p(\mathbb{C}),B\in M_{p,q}(\mathbb{C}),C\in M_q(\mathbb{C}),&\\
& & & A^*=-A, C^*=-C, \mbox{tr}(A)+\mbox{tr}(C)=0\Big\},&
\end{flalign*}
\begin{align*}
&\mathfrak{sp}(2p,\mathbb{R}):=
&\Big\{ 
\left(\begin{array}{cc}
A&B\\
B^*&\bar{A}
\end{array}\right) 
:A\in M_p(\mathbb{C}),B\in M_{p}(\mathbb{C}), A^*=-A, B^t=B\Big\}.
\end{align*}
Using these realisations we see immediately that $\mathfrak{sp}(2p,\mathbb{R})$ is a subalgebra of $\mathfrak{su}(p,p)$. We call this inclusion the \emph{standard inclusion}.

\subsection{Totally geodesic maps and holomorphicity}

We will from this point only work with Lie algebra homomorphisms. We will often abuse the notions slightly and talk about tight and holomorphic Lie algebra homomorphisms when we mean that the corresponding totally geodesic map has these properties.

Given Cartan decompositions $\mathfrak{g}_1=\mathfrak{k}_1+\mathfrak{p}_1$ and $\mathfrak{g}_2=\mathfrak{k}_2+\mathfrak{p}_2$ we can assume that a homomorphism $\rho\colon\mathfrak{g}_1\rightarrow\mathfrak{g}_2$ satisfies $\rho(\mathfrak{k}_1)\subset\mathfrak{k}_2$ and $\rho(\mathfrak{p}_1)\subset~\mathfrak{p}_2$.
Under this assumption it is easier to determine which homomorphisms are holomorphic. Recall that there is an element $Z_1$ (resp. $Z_2$) in the center of $\mathfrak{k}_1$ (resp. $\mathfrak{k}_2$) that correspond to the complex structure of $\mathcal{X}_1$ (resp. $\mathcal{X}_2$). A homomorphism $\rho\colon\mathfrak{g}_1\rightarrow\mathfrak{g}_2$ corresponds to a holomorphic map if $\rho\circ \mbox{ad}(Z_1)=\mbox{ad}(Z_2)\circ\rho$. This condition is called (H1).

A stronger condition for $\rho$ that plays an important role when working with holomorphic homomorphisms is if $\rho(Z_1)=Z_2$, we then say that $\rho$ satisifies (H2).
The following theorem from \cite{C6} tells us how to decompose holomorphic homomorphisms into managable parts. 

\begin{thm}[\cite{C6}]\label{iharastructure}
Let $\rho\colon\mathfrak{g}_1\rightarrow\mathfrak{g}_2$ be a homomorphism satisfying (H1). Then there exists a Hermitian regular subalgebra $\mathfrak{g}_3\subset\mathfrak{g}_2$ such that $\rho(\mathfrak{g}_1)\subset~\mathfrak{g}_3$ and such that the restricted homomorphism $\rho|\colon\mathfrak{g}_1\rightarrow\mathfrak{g}_3$ satisfies (H2).
\end{thm}
We will take a closer look at regular subalgebras in section \ref{regsa}.
It should be remarked that this decomposition of an (H1)-homomorphism is not necessarily unique.

\subsection{Root systems}
Let $V$ be a real finite dimensional vector space equipped with a positive definite inner product $\langle\cdot,\cdot\rangle$. A finite subset $\Delta\subset~V$ is called 
a \emph{root system} if the following axioms are satisfied:
\begin{enumerate}
\item $\Delta$ spans $V$,
\item for all $\alpha\in\Delta$ the reflection $s_\alpha\colon V \rightarrow V$, $s_\alpha(\beta)=\beta-2\frac{\langle \alpha,\beta\rangle}{\langle \alpha,\alpha\rangle}\alpha$ preserves $\Delta$,
\item for any pair $\alpha,\beta\in\Delta$ the number $a_{\alpha,\beta}:=2\frac{\langle \alpha,\beta\rangle}{\langle \alpha,\alpha\rangle}$ is an integer,
\item for any $\alpha\in\Delta$ the only multiples of $\alpha$ that belongs to $\Delta$ are $\alpha$ and $-\alpha$.
\end{enumerate}
The reflections $s_\alpha$, $\alpha\in\Delta$, forms a group called the \emph{Weyl group}.
We say that a subset $\Gamma\subset\Delta$ is a \emph{basis} for $\Delta$ if 
\begin{enumerate}
\item $\Gamma$ is a basis of $V$,
\item any $\beta\in\Delta$ can be written as a sum $\beta=\sum_{\alpha\in\Gamma}{n_\alpha \alpha}$ where either all $n_\alpha$ are positive integers or all $n_\alpha$ are negative integers.
\end{enumerate}
Any two bases can be mapped to each other by an element in the Weyl group.
The \emph{Dynkin diagram} of a root system is defined by putting a node for each basis element and connecting the nodes corresponding to $\alpha,\beta$ by $a_{\alpha,\beta}a_{\beta,\alpha}$ edges. If there are roots of different lengths an arrow is placed on edges connecting a longer root to a shorter one. 

From a basis $\Gamma$ follows a notion of positivity for $\Delta$. We say that a root $\beta=\sum_{\alpha\in\Gamma}{n_\alpha \alpha}$ is positive if the $n_\alpha$ are. 
We can also start in the other end and construct a basis from a notion of positivity for $\Delta$.
This positivity should satisfy:
\begin{enumerate}
\item for all $\alpha \in\Delta$ either $\alpha>0$ or $-\alpha>0$,
\item if $\alpha,\beta> 0$ and $\alpha+\beta\in\Delta$ then $\alpha+\beta>0$.
\end{enumerate}
Given a notion of positivity we say that a positive root is \emph{simple} if it can not be written as a sum of two other positive roots. The set of simple roots then forms a basis for $\Delta$, \cite[p.~178]{C3}. 
A practical way of defining positivity is by using regular vectors. A vector $v\in V$ is called \emph{regular} if $\langle \alpha,v\rangle \neq 0$ for all $\alpha\in\Delta$. We define a notion of positivity from $v$ by saying that a root $\alpha$ is positive if $\langle \alpha,v\rangle>0$. Given a notion of positivity there is a regular vector $v$ defining that notion of positivity, \cite{C3}.

Let $\Delta'$ be a subset of $\Delta$ and $V'=\mbox{span}_\mathbb{R}(\Delta')$. We say that $\Delta'$ is a \emph{subroot system} of $\Delta$ if $(\Delta',V')$ is a root system. A subset $\Gamma'$ of $\Delta$ will define a basis for some subroot system $\Delta^{'}$ if
\begin{enumerate}
\item $\Gamma'$ is linearly independent in $V$,
\item if $\alpha,\beta\in \Gamma'$ then $\alpha-\beta\not\in \Delta$.
\end{enumerate}
Such a set $\Gamma'$ is called a \emph{Dynkin $\Pi$-system}. We obtain a subroot system $\Delta(\Gamma')$ from $\Gamma'$ by
$\Delta(\Gamma')=\mbox{span}_\mathbb{Z}(\Gamma')\cap\Delta$.

\begin{lma}\label{positivebasis}
Let $\Delta$ be a root system with a fixed notion of positivity. For any subroot system $\Delta'\subset\Delta$ there exists a basis that consists of roots that are positive (with respect to the notion of positivity of $\Delta$).
\end{lma}
\begin{proof}
Let $v\in V$ be a regular vector defining the notion of positivity for $\Delta$. The orthogonal projection of $v$ onto $V'$ is then a regular vector $v'$ for the pair $(\Delta',V')$. The notion of positivity for $\Delta'$ defined by $v'$ agrees with that of $\Delta$. The set of simple roots now defines a basis for $\Delta'$ that consists of positive roots. 
\end{proof}

\subsection{Regular subalgebras}\label{regsa}
Let $\mathfrak{g}_\mathbb{C}$ be a complex semisimple Lie algebra with Killing form $\langle\cdot,\cdot\rangle$. Further let $\mathfrak{h}_\mathbb{C}\subset\mathfrak{g}_\mathbb{C}$ be a Cartan subalgebra and $\mathfrak{g}_\mathbb{C}=\mathfrak{h}_\mathbb{C}+\sum_{\alpha\in\Delta}{\mathfrak{g}_\alpha}$
a root space decomposition of $\mathfrak{g}_\mathbb{C}$. We define a real form $\mathfrak{h}_0$ of $\mathfrak{h}_\mathbb{C}$ by 
\[
\mathfrak{h}_0 := \{\, A \in \mathfrak{h}_\mathbb{C} :\, \alpha (A) \in \mathbb{R} \ 
\text{for any } \alpha \in \Delta \,\}.
\]
We can regard $\Delta$ as a subset of $\mathfrak{h}_0^{*}$. The Killing form of $\mathfrak{g}_\mathbb{C}$ restricts to a positive definite bilinear form on $\mathfrak{h}_0$ inducing a positive definite bilinear form on $\mathfrak{h}_0^{*}$. With respect to this inner product the pair $(\Delta, \mathfrak{h}_0^{*})$ is a root system.

A (semisimple) \emph{regular subalgebra} of $\mathfrak{g}_\mathbb{C}$ is a subalgebra derived from a subroot system $\Delta'$ of $\Delta$. The subroot system $\Delta'$ defines a subalgebra $\mathfrak{g}_\mathbb{C}(\Delta')=\mathfrak{h}'_\mathbb{C}+\sum_{\alpha\in\Delta'}{\mathfrak{g}_\alpha}$, where $\mathfrak{h}'_\mathbb{C}$ is the subalgebra of $\mathfrak{h}_\mathbb{C}$ spanned by the coroots of $\Delta'$.
We will sometimes describe the subroot system by a $\Pi$-system $\Gamma'$ and denote the corresponding regular subalgebra by $\mathfrak{g}_\mathbb{C}(\Gamma')$.

If $\mathfrak{g}_\mathbb{C}$ is the complexification of a Hermitian Lie algebra with Cartan decomposition $\mathfrak{g}=\mathfrak{k}+\mathfrak{p}$ we can choose a maximal abelian subalgebra 
$\mathfrak{h} \subset i\mathfrak{k}$.
As the complex structure $Z$ lies in the center of $\mathfrak{k}$ it will lie in the maximal abelian subalgebra $\mathfrak{h}$ of $\mathfrak{k}$.
The complexification $\mathfrak{h}_\mathbb{C}$ of $\mathfrak{h}$
is then a Cartan subalgebra of $\mathfrak{g}_\mathbb{C}$. Let $\mathfrak{g}_\mathbb{C}=\mathfrak{h}_\mathbb{C}+\sum_{\alpha\in\Delta}{\mathfrak{g}_\alpha}$ be a root space decomposition with respect to this Cartan subalgebra.
Since $\mbox{ad}(Z)(X)=0$ for any $X\in\mathfrak{k}$ and $\mbox{ad}(Z)^2(X)=-X$ for any $X\in\mathfrak{p}$ we can conclude that any $\mathfrak{g}_\alpha$ is contained in either $\mathfrak{k}_\mathbb{C}$ or $\mathfrak{p}_\mathbb{C}$. We call the root $\alpha$ compact in the first case and noncompact in the second. 
We equip $\Delta$ with a notion of positivity such that a noncompact root $\alpha$ is positive if $\alpha(Z)=i$.

\begin{thm}[\cite{C6}]\label{ihara}
Let $\mathfrak{g}$ be a Hermitan Lie algebra with Cartan decomposition $\mathfrak{g}=\mathfrak{k}+\mathfrak{p}$. Furher let $\mathfrak{h}\subset\mathfrak{k}$ be a maximal abelian subalgebra and $\mathfrak{g}_\mathbb{C}=\mathfrak{h}_\mathbb{C}+\sum_{\alpha\in\Delta}{\mathfrak{g}_\alpha}$ a root space decomposition of the complexification of $\mathfrak{g}$ with respect to the complexification of $\mathfrak{h}$.  
Let $\Delta$ be equipped with a notion of positivity as above and let $\Gamma'$ be a Dynkin $\Pi$-system satisfying further
\begin{enumerate}
\setcounter{enumi}{2}
\item each connected component of the Dynkin diagram of $\Gamma^{'}$ contains at most one positive noncompact root.\label{ma}
\end{enumerate}
Then the subalgebra $\mathfrak{g}(\Gamma'):=\mathfrak{g}_\mathbb{C}(\Gamma')\cap\mathfrak{g}$ is of Hermitian type and can be equipped with a unique complex structure such that the inclusion corresponds to a holomorphic map between the corresponding Hermitian symmetric spaces.
\end{thm}

It turns out that the condition (\ref{ma}) is superfluous. 

\begin{thm}
Let $\mathfrak{g}$ be a Hermitan Lie algebra with Cartan decomposition $\mathfrak{g}=\mathfrak{k}+\mathfrak{p}$. Furher let $\mathfrak{h}\subset\mathfrak{k}$ be a maximal abelian subalgebra and $\mathfrak{g}_\mathbb{C}=\mathfrak{h}_\mathbb{C}+\sum_{\alpha\in\Delta}{\mathfrak{g}_\alpha}$ a root space decomposition of the complexification of $\mathfrak{g}$ with respect to the complexification of $\mathfrak{h}$. 
Let $\Gamma'$ be a Dynkin $\Pi$-system.
Then the subalgebra $\mathfrak{g}(\Gamma'):=\mathfrak{g}_\mathbb{C}(\Gamma')\cap\mathfrak{g}$ is of Hermitian type and can be equipped with a unique complex structure such that the inclusion corresponds to a holomorphic map between the corresponding Hermitian symmetric spaces.
\end{thm}

\begin{proof}
Let $\Gamma'\subset\Delta$ be a $\Pi$-system and $\Delta(\Gamma')\subset\Delta$ the subroot system it generates. By Lemma \ref{positivebasis} we can find a basis $\Gamma''$ of $\Delta(\Gamma')$ consisting of only positive roots. 
We claim that $\Gamma''$ satisfies (\ref{ma}). 

Suppose that the Dynkin diagram of $\Gamma''$ has more than one noncompact root in one of its components. Let $\beta_1, \beta_2$ be two such noncompact roots and $\alpha_1,...,\alpha_n$ a string of compact roots connecting them. Then $\beta=\beta_1+\sum_1^n{\alpha_i}+\beta_2$ lies in $\Delta(\Gamma'')$. But both $\beta_1$ and $\beta_2$ are positive noncompact roots which implies that $\beta(Z)=2i$. This is a contradiction.

Since $\Delta(\Gamma')=\Delta(\Gamma'')$ we have $\mathfrak{g}(\Gamma')=\mathfrak{g}(\Gamma'')$. Now the result follows by Theorem \ref{ihara}. 
\end{proof}
We call a subalgebra $\mathfrak{g}(\Gamma')\subset\mathfrak{g}$ constructed in this way a \emph{Hermitian regular subalgebra}.

\subsection{Weighted Dynkin Diagrams}
\newcommand{\Map}{\mathop{\mathrm{Map}}\nolimits}

Let $\mathfrak{g}_\mathbb{C}$ be a complex semisimple Lie algebra and fix a Cartan subalgebra $\mathfrak{h}_\mathbb{C}\subset\mathfrak{g}_\mathbb{C}$.
Let $\mathfrak{g}_\mathbb{C}=\mathfrak{h}_\mathbb{C}+\sum_\Delta\mathfrak{g}_\alpha$
be the root space decomposition with respect to $\mathfrak{h}_\mathbb{C}$.
We again consider the real form $\mathfrak{h}_0$ of $\mathfrak{h}_\mathbb{C}$,  
\[
\mathfrak{h}_0 := \{\, H \in \mathfrak{h}_\mathbb{C} :\, \alpha (H) \in \mathbb{R} \ 
\text{for any } \alpha \in \Delta \,\}.
\]
and regard $\Delta$ as a subset of $\mathfrak{h}_0^*$.
We fix a basis $\Gamma$ for $\Delta$.
Then a closed Weyl chamber 
\[
\mathfrak{h}_+ := 
\{\, H \in \mathfrak{h}_0 :\, \alpha (H) \geq 0 \ 
\text{for any } \alpha \in \Gamma) \,\}
\]
is a fundamental domain of $\mathfrak{h}_0$ 
for the action of the Weyl group $W$.
Then, for each $H \in \mathfrak{h}_0$, we can define a map 
\begin{align*}
\Psi_H \colon \Gamma \rightarrow \mathbb{R},\ \alpha \mapsto \alpha(H).
\end{align*}
We call $\Psi_H$ the \emph{weighted Dynkin diagram} corresponding to $H \in \mathfrak{h}_0$,
and $\alpha(H)$ the weight on a node $\alpha \in \Gamma$ of the weighted Dynkin diagram.
Since $\Gamma$ is a basis of $\mathfrak{h}_0^*$, the correspondence 
\begin{align*}
\Psi\colon \mathfrak{h}_0 \rightarrow \Map(\Gamma,\mathbb{R}),\ H \mapsto \Psi_H \label{eq:psi}
\end{align*}
is a linear isomorphism between real vector spaces. 
In particular, $\Psi$ is bijective.
Furthermore, 
\[
\Psi|_{\mathfrak{h}_+} \colon \mathfrak{h}_+ \rightarrow \Map(\Gamma,\mathbb{R}_{\geq 0}),\ H \mapsto \Psi_H
\]
is also bijective.

Let us fix a complex Lie algebra homomorphism 
$\iota_\mathbb{C} \colon \mathfrak{sl}(2,\mathbb{C}) \rightarrow \mathfrak{g}_\mathbb{C}$.
It is known that there exists a unique element $A_\iota$ in $\mathfrak{h}_+$ such that 
the adjoint orbit in $\mathfrak{g}_\mathbb{C}$ through $A_\iota$ meets $\iota_\mathbb{C}(\mathfrak{sl}(2,\mathbb{C}))$ 
at $\iota_\mathbb{C} \begin{pmatrix} 1& 0 \\ 0 & -1 \end{pmatrix}$.
We define the \emph{weighted Dynkin diagram of 
the homomorphism} $\iota \colon \mathfrak{sl}(2,\mathbb{C}) \rightarrow \mathfrak{g}_\mathbb{C}$ 
to be the weighted Dynkin diagram corresponding to $A_{\iota} \in \mathfrak{h}_+$.
The weighted Dynkin diagram determines a homomorphism up to equivalence:
\begin{thm}[\cite{C9}]\label{wdd}
Two homomorphisms $\rho_\mathbb{C},\eta_\mathbb{C}\colon\mathfrak{sl}(2,\mathbb{C}) \rightarrow \mathfrak{g}_\mathbb{C}$ are equivalent if and only if their weighted Dynkin diagrams coincide. Further, the entries of any Dynkin diagram belongs to the set $\{ 0,1,2\}$.
\end{thm}
Weighted Dynkin diagrams will be presented by the Dynkin diagram of $\Gamma$ with the number $\alpha(A_\iota)$ next to each $\alpha\in\Gamma$.

\subsection{Properties of tight homomorphisms}
We will not deal with criterions for tightness directly but rather use previous classification results and known properties of tight maps to derive contradictions to the existence of nonholomorphic tight maps. We begin with a lemma that deals with compositions of tight maps.

\begin{lma}\label{compose}
Let $\rho\colon\mathfrak{g}_1\rightarrow\mathfrak{g}_2$ and $\eta\colon\mathfrak{g}_2\rightarrow\mathfrak{g}_3$ be homomorphisms between Hermitian Lie algebras where $\mathfrak{g}_2$ is simple. Then $\eta\circ \rho$ is tight if and only if $\eta$ and $\rho$ are tight.
\end{lma}
For a proof see \cite{C31}.
We will also need to know about the nonholomorphic tight homomorphisms of $\mathfrak{su}(1,1)$. These were classified in \cite{C8}:
\begin{thm}\label{biwclass}
Let $\rho\colon\mathfrak{su}(1,1)\rightarrow\mathfrak{g}$ be a tight homomorphism. Then $\rho$ factors as $\rho=\iota\circ\sum{\rho_i}\colon\mathfrak{su}(1,1)\rightarrow\oplus\mathfrak{sp}(2n_i,\mathbb{R})\rightarrow\mathfrak{g}$. Here the $\rho_i$ are (tight) irreducible representations $\rho_i\colon \mathfrak{su}(1,1)\rightarrow\mathfrak{sp}(2n_i,\mathbb{R})$ and $\iota$ is a tight and holomorphic embedding of $\oplus\mathfrak{sp}(2n_i,\mathbb{R})$ into $\mathfrak{g}$.
The homomorphism $\rho$ is nonholomorphic if and only if some $n_i$ is greater than one.
\end{thm}

We will also need the following structure theorem from \cite{C8}:
\begin{lma}\label{tube3} 
Let $\rho\colon\mathfrak{g}_1\rightarrow\mathfrak{g}_2$ be an injective tight homomorphism. If $\mathfrak{g}_2$ is of tube type then $\mathfrak{g}_1$ must be of tube type as well.
\end{lma}

\section{Proof of the Main Theorem}

Our proof uses two different methods for the two different exceptional Hermitian Lie algebras. We begin by doing some of the more technical parts in the following two lemmas.

\begin{lma}\label{mainlemma1}
Let $\rho\colon\mathfrak{su}(1,1)\rightarrow\mathfrak{e}_{6(-14)}$ be a tight nonholomorphic homomorphism. Then the weighted Dynkin diagram of the complexification $\rho_\mathbb{C}\colon\mathfrak{sl}(2,\mathbb{C})\rightarrow\mathfrak{e}_{6,\mathbb{C}}$ is:

\begin{picture}(50,47)
\multiput(5,31)(30,0){5}{\circle{5}}
\multiputlist(20,31)(30,0)%
{{\line(10,0){25}},{\line(10,0){25}},{\line(10,0){25}},{\line(10,0){25}}}
\multiputlist(5,41)(30,0){$1$,$0$,$0$,$0$,$1$}
\put(65,1){\circle{5}}
\put(65,4){\line(0,10){25}}
\put(70,0){$2$}
\end{picture}

\end{lma}

\begin{lma}\label{mainlemma2}
Let $\rho\colon\mathfrak{sp}(2n,\mathbb{R})\rightarrow\mathfrak{e}_{7(-25)}$ be a homomorphism where $n=2$ or $3$. Then there exists a simple proper Hermitian regular subalgebra of $\mathfrak{e}_{7(-25)}$ that contains the image of $\mathfrak{sp}(2n,\mathbb{R})$.
\end{lma}
Let us remark that Lemma \ref{mainlemma2} is in no way a corollary of Theorem~\ref{iharastructure} as we have no assumption of holomorphicity.

\begin{proof}[Proof of Lemma \ref{mainlemma1}]
The proof consists of two parts. In the first part we show that $\rho_\mathbb{C}$ is unique up to equivalence and give a concrete description of $\rho_\mathbb{C}$. In the second part we calculate the weighted Dynkin diagram of it.

By Theorem \ref{biwclass} a tight homomorphism $\rho\colon\mathfrak{su}(1,1)\rightarrow\mathfrak{e}_{6(-14)}$ factors as 
$\iota\circ\rho'\colon\mathfrak{su}(1,1)\rightarrow\oplus \mathfrak{sp}(2n_i,\mathbb{R})\rightarrow\mathfrak{e}_{6(-14)}$. Here $\rho'=\sum\rho_i$ is a sum of irreducible representations and $\iota$ is a holomorphic embedding. Since $\mathfrak{e}_{6(-14)}$ is of real rank two there are three possibilities for $\oplus\mathfrak{sp}(2n_i,\mathbb{R})$, either $\mathfrak{sp}(2,\mathbb{R})$, $\mathfrak{sp}(2,\mathbb{R})\oplus \mathfrak{sp}(2,\mathbb{R})$ or $\mathfrak{sp}(4,\mathbb{R})$. Since $\rho$ is assumed to be nonholomorphic we can conclude, again by Theorem \ref{biwclass}, that $\oplus\mathfrak{sp}(2n_i,\mathbb{R})=\mathfrak{sp}(4,\mathbb{R})$.
Next we consider the possibilities for $\iota$. By Theorem \ref{iharastructure} $\iota$ factors as $\iota=r\circ h\colon\mathfrak{sp}(4,\mathbb{R})\rightarrow\mathfrak{g}\rightarrow\mathfrak{e}_{6(-14)}$ where $h$ is an $(H2)$-homomorphism and $r$ is 
the inclusion of a Hermitian regular subalgebra. 
By a simple rank argument we can conclude that $\mathfrak{g}$ must be simple and of real rank two. 
As its complexification $\mathfrak{g}_\mathbb{C}$ is a regular subalgebra of $\mathfrak{e}_{6,\mathbb{C}}$ it can not have roots of different lengths.

We first analyse the possibilities for $h\colon\mathfrak{sp}(4,\mathbb{R})\rightarrow\mathfrak{g}$. 
The first possibility is the identity homomorphism into $\mathfrak{g}=\mathfrak{sp}(4,\mathbb{R})$. 
We can dismiss this possibility since $\mathfrak{sp}(4,\mathbb{C}) = \mathfrak{sp}(4,\mathbb{R})_{\mathbb{C}}$ can not be a regular subalgebra of $\mathfrak{e}_{6,\mathbb{C}}$ as it has roots of different lengths.

The second possibility is the standard inclusion into $\mathfrak{g}=\mathfrak{su}(2,2)$, \cite{C7}. 
Recall that $\mathfrak{sp}(4,\mathbb{R})$ is isomorphic to $\mathfrak{so}(3,2)$. This gives us a third possibility, an (H2)-homomorphism $\mathfrak{so}(3,2)\rightarrow\mathfrak{so}(q,2)$, $q>3$, as defined in \cite[p.~294]{C6}.
If $q$ is odd we again get roots of different lengths, we can thus conclude that $q$ must be even. This homomorphism decomposes as $\mathfrak{so}(3,2)\rightarrow\mathfrak{so}(4,2)\rightarrow\mathfrak{so}(q,2)$ where the second part is the inclusion of a regular subalgebra \cite[p.~294]{C6}. We choose to view the second part as part of the homomorphism $r$. We are thus back to the second case since $\mathfrak{so}(4,2)$ is isomorphic to $\mathfrak{su}(2,2)$.
In conclusion, every holomorphic homomorphism $\iota\colon\mathfrak{sp}(4,\mathbb{R})\rightarrow\mathfrak{e}_{6(-14)}$ factors as $\iota=r\circ h\colon\mathfrak{sp}(4,\mathbb{R})\rightarrow\mathfrak{su}(2,2)\rightarrow\mathfrak{e}_{6(-14)}$ with $h$ the holomorphic standard inclusion and $r$ some inclusion homomorphism of $\mathfrak{su}(2,2)$ as a Hermitian regular subalgebra of $\mathfrak{e}_{6(-14)}$.

Next we analyse the possibilities for $r\colon\mathfrak{su}(2,2)\rightarrow \mathfrak{e}_{6(-14)}$. Let $\Delta$ denote the root system of $\mathfrak{e}_{6,\mathbb{C}}$. We index the simple roots as in the Dynkin diagram below where $\alpha_1$ denotes the unique positive noncompact simple root.\\

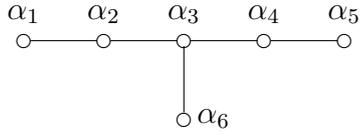
\begin{picture}(50,37)
\multiput(5,31)(30,0){5}{\circle{5}}
\multiputlist(20,31)(30,0)%
{{\line(10,0){25}},{\line(10,0){25}},{\line(10,0){25}},{\line(10,0){25}}}
\multiputlist(5,41)(30,0){$\alpha_{1}$,$\alpha_{2}$,$\alpha_{3}$,$\alpha_{4}$,$\alpha_{5}$}
\put(65,1){\circle{5}}
\put(65,4){\line(0,10){25}}
\put(70,0){$\alpha_6$}
\end{picture}

Any regular subalgebra fits into a chain of maximal (with respect to inclusion) regular subalgebras. In \cite[p.~283-291]{C6} we find tables of the maximal Hermitian regular subalgebras. Investigating these tables we see that there are four possible chains of inclusions of $\mathfrak{su}(2,2)$ into $\mathfrak{e}_{6(-14)}$, namely

\begin{align*}
r_1\colon\mathfrak{su}(2,2)\hookrightarrow\mathfrak{su}(3,2)\hookrightarrow \mathfrak{su}(4,2)\hookrightarrow\mathfrak{e}_{6(-14)},\\
r_2\colon\mathfrak{su}(2,2)\hookrightarrow\mathfrak{su}(3,2)\hookrightarrow \mathfrak{so}^{*}(10)\hookrightarrow\mathfrak{e}_{6(-14)},\\
r_3\colon\mathfrak{su}(2,2)\hookrightarrow \mathfrak{so}(8,2)\hookrightarrow\mathfrak{e}_{6(-14)},\\
r_4\colon\mathfrak{su}(2,2)= \mathfrak{so}(4,2)\hookrightarrow \mathfrak{so}(6,2)\hookrightarrow \mathfrak{so}(8,2)\hookrightarrow\mathfrak{e}_{6(-14)}.
\end{align*}
These correspond to the following inclusions of root systems
\allowdisplaybreaks
\begin{align*}
\Delta(\{\alpha_2+2\alpha_3+2\alpha_4+\alpha_5+\alpha_6,\alpha_1,\alpha_2\})&\subset \\
\Delta(\{\alpha_2+2\alpha_3+2\alpha_4+\alpha_5+\alpha_6,\alpha_1,\alpha_2,\alpha_3\})&\subset\\
\Delta(\{\alpha_2+2\alpha_3+2\alpha_4+\alpha_5+\alpha_6,\alpha_1,\alpha_2,\alpha_3,\alpha_6\})&\subset\Delta,\\
\Delta(\{\alpha_2+2\alpha_3+2\alpha_4+\alpha_5+\alpha_6,\alpha_1,\alpha_2\})&\subset \\
\Delta(\{\alpha_2+2\alpha_3+2\alpha_4+\alpha_5+\alpha_6,\alpha_1,\alpha_2,\alpha_3\})&\subset \\
\Delta(\{\alpha_1,\alpha_2,\alpha_3,\alpha_4, \alpha_3+\alpha_4+\alpha_5+\alpha_6\})&\subset\Delta,\\
\Delta(\{\alpha_2+2\alpha_3+\alpha_4+\alpha_6,\alpha_1,\alpha_2\})&\subset \\
\Delta(\{\alpha_1,\alpha_2,\alpha_3,\alpha_4, \alpha_6\})&\subset\Delta,\\
\Delta(\{\alpha_1,\alpha_2,\alpha_2+2\alpha_3+\alpha_4+\alpha_6\})&\subset \\
\Delta(\{\alpha_1,\alpha_2,\alpha_3, \alpha_3+\alpha_4+\alpha_6\})&\subset \\
\Delta(\{\alpha_1,\alpha_2, \alpha_3,\alpha_4, \alpha_6\})&\subset\Delta.
\end{align*}

From the subroot systems we see immediatly that $r_{1,\mathbb{C}}$ and $r_{2,\mathbb{C}}$ coincide with each other. Applying the reflection  $s_{\alpha_4+\alpha_5}$ we see that $r_{3,\mathbb{C}}$ and $r_{4,\mathbb{C}}$ are equivalent to $r_{1,\mathbb{C}}$ as
\begin{align*}
s_{\alpha_4+\alpha_5}(\alpha_2+2\alpha_3+\alpha_4+\alpha_6)&=\alpha_2+2\alpha_3+2\alpha_4+\alpha_5+\alpha_6,\\
s_{\alpha_4+\alpha_5}(\alpha_1)&=\alpha_1,\\
s_{\alpha_4+\alpha_5}(\alpha_2)&=\alpha_2.
\end{align*} 
We have now seen that $\rho' , h$ and $r_\mathbb{C}$ are unique up to equivalence.
We can thus conclude that $\rho_\mathbb{C}$ is unique up to equivalence.

The homomorphism $\rho'$ is given by\\
$\left(\begin{array}{cc}
a&b\\
\bar{b}&-a
\end{array}\right) \mapsto$
 $\left(\begin{array}{cccc}
3a&0&0&\sqrt{3}b\\
0&-a&2\bar{b}&\sqrt{3}b\\
\sqrt{3}\,\bar{b}&2b&-3a&0\\
0&\sqrt{3}\,\bar{b}&0&a
\end{array}\right) $.\\
The homomorphism $h$ is just the inclusion homomorphism.
Let $\mathfrak{h}_\mathbb{C}\subset\mathfrak{sl}(4,\mathbb{C})=\mathfrak{su}(2,2)_\mathbb{C}$ be the Cartan subalgebra of diagonal matrices. The coroots associated to the root space decomposition of $\mathfrak{sl}(4,\mathbb{C})$ with respect to $\mathfrak{h}_\mathbb{C}$ are
\begin{eqnarray*}
H_3=\mbox{diag}(1,-1,0,0)\\
H_1=\mbox{diag}(0,1,-1,0)\\
H_2=\mbox{diag}(0,0,1,-1).
\end{eqnarray*}

Let $\{H'_i\}_1^6$ denote the coroots of $\mathfrak{e}_{6,\mathbb{C}}$.
From the subroot systems we see that $r_\mathbb{C}$ maps $H_3$ to $H'_2+2H'_3+2H'_4+H'_5+H'_6$, $H_1$ to $H'_1$ and $H_2$ to $H'_2$.
Let $H=\left(\begin{array}{cc}
1&0\\
0&-1
\end{array}\right)$,
we get 
\begin{align*}
\rho_\mathbb{C}(H)&=r_\mathbb{C}\circ h_\mathbb{C}\circ\rho'_\mathbb{C}(H)=r_\mathbb{C}(3H_3+2H_1-H_2)\\
&=2H'_1+2H'_2+6H'_3+6H'_4+3H'_5+3H'_6.
\end{align*} 
To calculate the weighted Dynkin diagram of $\rho_\mathbb{C}$ we first calculate all $\alpha_i(\rho_\mathbb{C}(H))$ and then apply reflections from the Weyl group of $\mathfrak{e}_{6,\mathbb{C}}$ until all entries are positive.

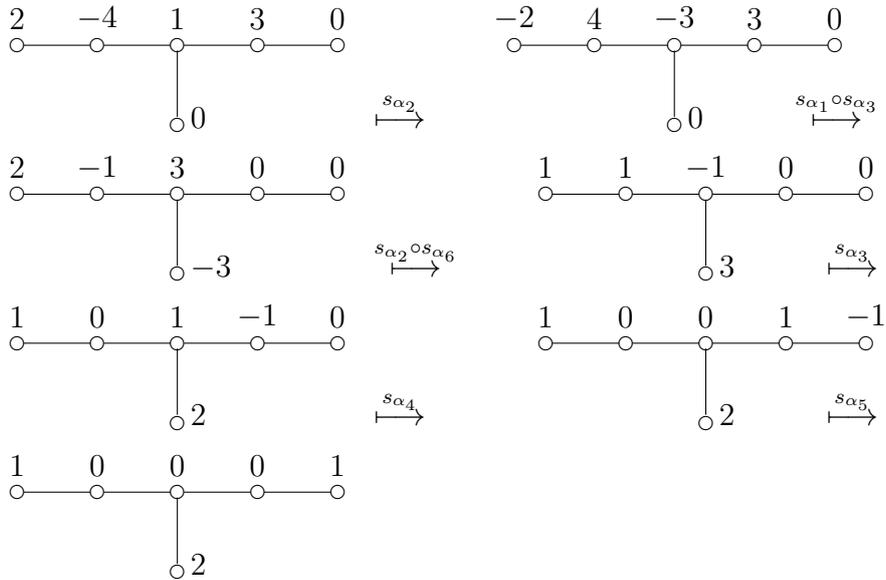
\begin{align*}
&\begin{picture}(50,47)
\multiput(5,31)(30,0){5}{\circle{5}}
\multiputlist(20,31)(30,0)%
{{\line(10,0){25}},{\line(10,0){25}},{\line(10,0){25}},{\line(10,0){25}}}
\multiputlist(5,41)(30,0){$2$,$-4$,$1$,$3$,$0$}
\put(65,1){\circle{5}}
\put(65,4){\line(0,10){25}}
\put(70,0){$0$}
\end{picture} 
\hspace{3cm} \stackrel{s_{\alpha_2}}{\longmapsto} 
&\begin{picture}(50,47)
\multiput(5,31)(30,0){5}{\circle{5}}
\multiputlist(20,31)(30,0)%
{{\line(10,0){25}},{\line(10,0){25}},{\line(10,0){25}},{\line(10,0){25}}}
\multiputlist(5,41)(30,0){$-2$,$4$,$-3$,$3$,$0$}
\put(65,1){\circle{5}}
\put(65,4){\line(0,10){25}}
\put(70,0){$0$}
\end{picture}  
\hspace{2cm}\stackrel{s_{\alpha_1}\circ s_{\alpha_3}}{\longmapsto}\\
&\begin{picture}(50,47)
\multiput(5,31)(30,0){5}{\circle{5}}
\multiputlist(20,31)(30,0)%
{{\line(10,0){25}},{\line(10,0){25}},{\line(10,0){25}},{\line(10,0){25}}}
\multiputlist(5,41)(30,0){$2$,$-1$,$3$,$0$,$0$}
\put(65,1){\circle{5}}
\put(65,4){\line(0,10){25}}
\put(70,0){$-3$}
\end{picture}  
\hspace{3cm} \stackrel{s_{\alpha_2}\circ s_{\alpha_6}}{\longmapsto}
&\begin{picture}(50,47)
\multiput(5,31)(30,0){5}{\circle{5}}
\multiputlist(20,31)(30,0)%
{{\line(10,0){25}},{\line(10,0){25}},{\line(10,0){25}},{\line(10,0){25}}}
\multiputlist(5,41)(30,0){$1$,$1$,$-1$,$0$,$0$}
\put(65,1){\circle{5}}
\put(65,4){\line(0,10){25}}
\put(70,0){$3$}
\end{picture}  
\hspace{2cm}\stackrel{s_{\alpha_3}}{\longmapsto}\\
&\begin{picture}(50,47)
\multiput(5,31)(30,0){5}{\circle{5}}
\multiputlist(20,31)(30,0)%
{{\line(10,0){25}},{\line(10,0){25}},{\line(10,0){25}},{\line(10,0){25}}}
\multiputlist(5,41)(30,0){$1$,$0$,$1$,$-1$,$0$}
\put(65,1){\circle{5}}
\put(65,4){\line(0,10){25}}
\put(70,0){$2$}
\end{picture}  
\hspace{3cm} \stackrel{s_{\alpha_4}}{\longmapsto}
&\begin{picture}(50,47)
\multiput(5,31)(30,0){5}{\circle{5}}
\multiputlist(20,31)(30,0)%
{{\line(10,0){25}},{\line(10,0){25}},{\line(10,0){25}},{\line(10,0){25}}}
\multiputlist(5,41)(30,0){$1$,$0$,$0$,$1$,$-1$}
\put(65,1){\circle{5}}
\put(65,4){\line(0,10){25}}
\put(70,0){$2$}
\end{picture}  
\hspace{2cm}\stackrel{s_{\alpha_5}}{\longmapsto}\\
&\begin{picture}(50,47)
\multiput(5,31)(30,0){5}{\circle{5}}
\multiputlist(20,31)(30,0)%
{{\line(10,0){25}},{\line(10,0){25}},{\line(10,0){25}},{\line(10,0){25}}}
\multiputlist(5,41)(30,0){$1$,$0$,$0$,$0$,$1$}
\put(65,1){\circle{5}}
\put(65,4){\line(0,10){25}}
\put(70,0){$2$}
\end{picture}
\end{align*}
\end{proof}

\begin{proof}[Proof of Lemma \ref{mainlemma2}]
We fix the following notation. Let $\mathfrak{g}'=\mathfrak{sp}(4,\mathbb{R})$ or $\mathfrak{sp}(6,\mathbb{R})$ and $\mathfrak{g}=\mathfrak{e}_{7(-25)}$ with Cartan decompositions $\mathfrak{g}'=\mathfrak{k}'+\mathfrak{p}'$ and  $\mathfrak{g}=\mathfrak{k}+\mathfrak{p}$.
We have complexifications $\mathfrak{g}'_\mathbb{C }$ and $\mathfrak{g}_\mathbb{C }$ and compact real forms
$\mathfrak{u}'=\mathfrak{k}'+i\mathfrak{p}'$ and  $\mathfrak{u}=\mathfrak{k}+i\mathfrak{p}$.
We denote by $\theta',\theta$ conjugation of $\mathfrak{g}'_\mathbb{C}$, $\mathfrak{g}_\mathbb{C}$ with respect to $\mathfrak{g}'$, $\mathfrak{g}$ and by $\tau'$, $\tau$ conjugation with respect to the compact real forms.
We assume that $\rho\colon \mathfrak{g}'\rightarrow\mathfrak{g}$ respects Cartan decompositions, this is equivalent to that the complexified homomorphism $\rho_\mathbb{C}$ satisfies $\rho_\mathbb{C}\circ\tau'=\tau\circ\rho_\mathbb{C}$ and $\rho_\mathbb{C}\circ\theta'=\theta\circ\rho_\mathbb{C}$.
An important fact we will need is that for $\mathfrak{g}'_\mathbb{C}=\mathfrak{sp}(4,\mathbb{C})$ or $\mathfrak{sp}(6,\mathbb{C})$
and any $\rho_\mathbb{C}\colon\mathfrak{g}'_\mathbb{C}\rightarrow\mathfrak{e}_{7,\mathbb{C}}$ the centralizer of $\rho_\mathbb{C}(\mathfrak{g}'_\mathbb{C})$ in $\mathfrak{e}_{7,\mathbb{C}}$ has dimension at least four, \cite[ pp.~197,201]{C9}.

Let $\mathfrak{z}_\mathbb{C}$ denote the centralizer of $\rho_\mathbb{C}(\mathfrak{g}'_\mathbb{C})\subset\mathfrak{g}_\mathbb{C}$. Given $X\in\mathfrak{z}_\mathbb{C}$, we have
\begin{equation*}
[\theta(X),\rho(Y)]=\theta([X,\theta(\rho(Y))])=\theta([X,\rho(\theta'(Y))])=0
\end{equation*} 
for all $\rho(Y)\in\rho(\mathfrak{g}'_\mathbb{C})$, i.e. $\theta(\mathfrak{z}_\mathbb{C})=\mathfrak{z}_\mathbb{C}$.
The same calculation goes through with $\theta$ replaced by $\tau$. Let $\mathfrak{z}:=\{ Z+\theta(Z) :\, Z\in\mathfrak{z}_\mathbb{C}\}$, then $\mathfrak{z}$ is a subalgebra of $\mathfrak{g}$ and $\mbox{dim}_\mathbb{R}(\mathfrak{z})=\mbox{dim}_\mathbb{C}(\mathfrak{z}_\mathbb{C})$.

We claim that $\mathfrak{z}$ can not be contained in $\mathfrak{p}$.
We know that $\mathfrak{z}$ is a subalgebra of $\mathfrak{g}$ and that the only subalgebras of $\mathfrak{p}$ are abelian ones.
If $\mathfrak{z}$ was contained in $\mathfrak{p}$ it would be an abelian subalgebra of dimension at least four which contradicts that 
the real rank of $\mathfrak{e}_{7(-25)}$ is three.

Take a nonzero element $X\in\mathfrak{z}$ that does not lie in $\mathfrak{p}$. Then $H_0:=X+(\theta \tau)(X)$ lies in $\mathfrak{k}$ and is nonzero. Let $\mathfrak{h}$ be a maximal abelian subalgebra of $\mathfrak{k}$ containing $H_0$.
We want to construct a minimal regular subalgebra of $\mathfrak{g}_\mathbb{C}$ with respect to $\mathfrak{h}_\mathbb{C}$ that contains $\rho(\mathfrak{g}'_\mathbb{C})$. Since $[\mathfrak{p}'_\mathbb{C},\mathfrak{p}'_\mathbb{C}]=\mathfrak{k}'_\mathbb{C}$ it is sufficient to find a regular subalgebra containing $\rho(\mathfrak{p}'_\mathbb{C})$.

Let $\mathfrak{g}_\mathbb{C}= \mathfrak{h}_\mathbb{C}+\sum_{\alpha\in\Delta}{\mathfrak{g}_\alpha}$ 
be the root space decomposition of $\mathfrak{g}_\mathbb{C}$ with respect to $\mathfrak{h}_\mathbb{C}$.
We have that 
$\rho(\mathfrak{p}'_\mathbb{C})=\mbox{span}_\mathbb{C}\{Y_1,...,Y_m\}$ 
for some $Y_i$. Since $\rho(\mathfrak{p}'_\mathbb{C})\subset\mathfrak{p}_\mathbb{C}\subset\sum{\mathfrak{g}_\alpha}$, the $Y_i$ are of the form  $Y_i=\sum_{\alpha\in\Delta_i}{c_{i\alpha}X_\alpha}$, where $X_\alpha\in\mathfrak{g}_\alpha$ and the subsets $\Delta_i$ are chosen such that all $c_{i\alpha}\neq 0$.

Let $\Delta''=\cup_i \Delta_i$ and 
$\Delta'=\mbox{span}_\mathbb{Z}(\Delta'')\cap\Delta$. 
Then $\Delta'$ is a subroot system of $\Delta$. The regular subalgebras $\mathfrak{g}_\mathbb{C}(\Delta')$ respectively $\mathfrak{g}(\Delta')$ contain $\rho(\mathfrak{g}_\mathbb{C})$ respectively $\rho(\mathfrak{g})$.

Since $0=[H_0,Y_i]=\sum_{\alpha\in\Delta_i}{c_{i\alpha}\alpha(H_0)X_\alpha}$ for all $i$ we have that $\alpha(H_0)=0$ for all $\alpha\in\Delta'$.
Hence $[H_0,\mathfrak{g}_\mathbb{C}(\Delta')]=0$. Since $\mathfrak{g}(\Delta')$ is semisimple this implies that $H_0\not\in\mathfrak{g}_\mathbb{C}(\Delta')$. Hence $\mathfrak{g}(\Delta')$  is a proper regular Hermitian subalgebra 
of $\mathfrak{g}$ containing $\rho(\mathfrak{g}')$. 

It remains to prove that $\mathfrak{g}(\Delta')$ is simple. 
Let $\mathfrak{g}(\Delta')=\oplus\mathfrak{g}_i$ be a decomposition into simple Lie algebras. Since $\mathfrak{g}(\Delta')$ is the smallest Hermitian regular subalgebra containing $\rho(\mathfrak{g}')$, $\rho$ is a sum of injective homomorphisms $\rho_i\colon\mathfrak{g}'\rightarrow\mathfrak{g}_i$. This implies that $\mbox{rank}(\mathfrak{g}_i)\geq\mbox{rank}(\mathfrak{g}')\geq 2$. On the other hand, since $\mathfrak{g}(\Delta')$ is a subalgebra of $\mathfrak{e}_{7(-25)}$ we have that $\sum{\mbox{rank}(\mathfrak{g}_i)}\leq 3$. 
The only way both of these inequalities are satisfied is if $\mathfrak{g}(\Delta')$ is simple.
\end{proof}

\begin{proof}[Proof of Theorem \ref{main3}]
Let us first consider possible tight nonholomorphic homomorphisms $\rho\colon\mathfrak{g}\rightarrow\mathfrak{e}_{6(-14)}$. That  $\mathfrak{g}$ is of rank one and not isomorphic to $\mathfrak{su}(1,1)$ implies that $\mathfrak{g}$ is isomorphic to $\mathfrak{su}(n,1)$, $n\geq 2$.

Let $\iota\colon\mathfrak{su}(1,1)\rightarrow\mathfrak{su}(n,1)$ be the tight and holomorphic homomorphism
\begin{equation*}
\left(\begin{array}{cc}
ai&z\\
\bar{z}&-ai
\end{array}\right)\mapsto
\left(\begin{array}{cccc}
0&0&0&0\\
0&ai&0&z\\
0&0&0&0\\
0&\bar{z}&0&-ai
\end{array}\right),
\end{equation*}
where the zeros in the first row and column are block matrices. We get the complexification
$\iota_\mathbb{C}\colon\mathfrak{sl}(2,\mathbb{C})\rightarrow\mathfrak{sl}(n+1,\mathbb{C})$,
\begin{equation*}
\left(\begin{array}{cc}
a&b\\
c&-a
\end{array}\right)\mapsto
\left(\begin{array}{cccc}
0&0&0&0\\
0&a&0&b\\
0&0&0&0\\
0&c&0&-a
\end{array}\right)
\end{equation*}
We also define the homomorphism $\phi_\mathbb{C}\colon\mathfrak{sl}(2,\mathbb{C})\rightarrow\mathfrak{sl}(n+1,\mathbb{C})$, 
\begin{equation*}
\left(\begin{array}{cc}
a&b\\
c&-a
\end{array}\right)\mapsto
\left(\begin{array}{cccc}
0&0&0&0\\
0&2a&b&0\\
0&2c&0&2b\\
0&0&c&-2a
\end{array}\right).
\end{equation*}

Suppose $\rho\colon\mathfrak{su}(n,1)\rightarrow\mathfrak{e}_{6(-14)}$ is tight and nonholomorphic. 
By Lemma \ref{compose} the composition $\rho\circ\iota$ is tight, and as a composition of a holomorphic map and a nonholomorphic map it is nonholomorphic. Hence by Lemma \ref{mainlemma1} $\rho_\mathbb{C}\circ\iota_\mathbb{C}$ has weighted Dynkin diagram: 

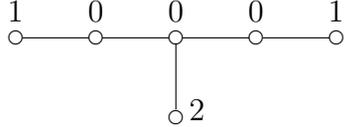
\begin{picture}(50,47)
\multiput(5,31)(30,0){5}{\circle{5}}
\multiputlist(20,31)(30,0)%
{{\line(10,0){25}},{\line(10,0){25}},{\line(10,0){25}},{\line(10,0){25}}}
\multiputlist(5,41)(30,0){$1$,$0$,$0$,$0$,$1$}
\put(65,1){\circle{5}}
\put(65,4){\line(0,10){25}}
\put(70,0){$2$}
\end{picture}

Since $\phi_\mathbb{C}(H)=2\iota_\mathbb{C}(H)$ the entries of the weighted Dynkin diagram of $\rho_\mathbb{C}\circ\phi_\mathbb{C}$ are two times those of $\rho_\mathbb{C}\circ\iota_\mathbb{C}$.  Hence  $\rho_\mathbb{C}\circ\phi_\mathbb{C}$ has weighted Dynkin diagram:

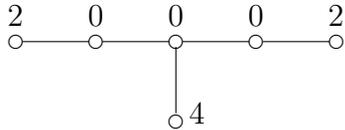
\begin{picture}(50,47)
\multiput(5,31)(30,0){5}{\circle{5}}
\multiputlist(20,31)(30,0)%
{{\line(10,0){25}},{\line(10,0){25}},{\line(10,0){25}},{\line(10,0){25}}}
\multiputlist(5,41)(30,0){$2$,$0$,$0$,$0$,$2$}
\put(65,1){\circle{5}}
\put(65,4){\line(0,10){25}}
\put(70,0){$4$}
\end{picture}

But by Theorem \ref{wdd} all entries of a weighted Dynkin diagram belongs to the set $\{0,1,2\}$. This is a contradiction, hence a tight nonholomorphic homomorphism $\rho\colon\mathfrak{su}(n,1)\rightarrow\mathfrak{e}_{6(-14)}$ can not exist.

 Next we assume that $\rho\colon\mathfrak{g}\rightarrow\mathfrak{e}_{7(-25)}$ is a nonholomorphic tight homomorphisms, $\mathfrak{g}\neq \mathfrak{su}(1,1)$. Since the real rank of $\mathfrak{e}_{7(-25)}$ is three we can assume that $\mbox{rank}(\mathfrak{g})\leq 3$. By Lemma \ref{tube3} it must be of tube type. This leaves us with seven possibilities: $\mathfrak{g}=\mathfrak{su}(2,2)$, $\mathfrak{su}(3,3)$, $\mathfrak{so}^*(8)$, $\mathfrak{so}^*(12)$, $\mathfrak{sp}(4,\mathbb{R})$, $\mathfrak{sp}(6,\mathbb{R})$ or $\mathfrak{so}(n,2)$, $n\geq 3$.
To each of these $\mathfrak{g}$ there exists a tight and holomorphic homomorphism $\iota\colon\mathfrak{sp}(4,\mathbb{R})\rightarrow\mathfrak{g}$ or $\iota\colon\mathfrak{sp}(6,\mathbb{R})\rightarrow\mathfrak{g}$, \cite{C31}. By Lemma \ref{mainlemma2} the image of $\rho\circ\iota$ is contained in a simple proper Hermitian regular subalgebra $\mathfrak{g}'\subset\mathfrak{e}_{7(-25)}$.
We thus get the following commutative diagram:

\centerline{\xymatrix{
\mathfrak{g}\ar[r]^\rho &\mathfrak{e}_{7(-25)} \\
\mathfrak{sp}(4/6,\mathbb{R})\ar[u]^{\iota}\ar[r]^{\rho|} &\mathfrak{g}'\ar[u]_{r} 
}}

Since $\rho$ and $\iota$ are tight Lemma \ref{compose} implies that $\rho\circ\iota=r\circ\rho|$ is tight. Applying the lemma again implies that $r$ and $\rho|$ are tight. By assumption $\rho$ is nonholomorphic and $\iota$ is holomorphic. Hence $\rho\circ\iota=r\circ \rho|$ is nonholomorphic. The inclusion homomorphism of the Hermitian regular subalgebra, $r$, is holomorphic, hence $\rho|$ must be nonholomorphic. Since $\mathfrak{g}'$ is a simple proper Hermitian subalgebra of $\mathfrak{e}_{7(-25)}$ it must either be isomorphic to a classical Hermitian Lie algebra or $\mathfrak{e}_{6(-14)}$. Summarizing, $\rho|\colon\mathfrak{sp}(4/6,\mathbb{R})\rightarrow\mathfrak{g}'$ is a nonholomorphic tight homomorphism where $\mathfrak{g}'$ is classical or equal to $\mathfrak{e}_{6(-14)}$. This contradicts Theorem \ref{oldmain}, hence there can not exist a tight and nonholomorphic homomorphism $\rho\colon\mathfrak{g}\rightarrow\mathfrak{e}_{7(-25)}$.
\end{proof}

\section{Appendix}

For the readers convenience we list here the Dynkin diagrams and Hermitian regular subalgebras of relevant Lie algebras. We follow the convention in \cite{C6} and let $\alpha_1$ denote the noncompact simple root. Together with each subalgebra is listed the $\Pi$-system defining it. In all root systems appearing below the roots are all of the same length. Hence, subalgebras $\mathfrak{g}'\subset\mathfrak{g}$ are tightly included if and only if $\mbox{rank}(\mathfrak{g}')=\mbox{rank}(\mathfrak{g}')$ by  \cite[Theorem 4.1]{C31}. \\

{\bf Maximal Hermitian regular subalgebras of $\mathfrak{su}(p,q)$}\\

\begin{picture}(50,43)
\multiput(5,31)(30,0){7}{\circle{5}}
\multiputlist(20,31)(30,0)%
{{\line(10,0){25}},{$\cdots$},{\line(10,0){25}},{\line(10,0){25}},{$\cdots$},{\line(10,0){25}}}
\multiputlist(5,41)(30,0){$\alpha_{q+1}$,$\alpha_{q+2}$,$\alpha_{q+p-1}$,$\alpha_{1}$,$\alpha_{2}$,$\alpha_{q-1}$,$\alpha_{q}$}
\end{picture}
\begin{align*}
&\gamma=\alpha_1+...+\alpha_{p+q-1}
\end{align*}
\begin{align*}
&\mathfrak{su}(l,q) ,\,1\leq l <p,\\
& \{\alpha_{p+q-l},\,\alpha_{p+q-l+1},...,\alpha_{p+q-1},\alpha_1,\alpha_2,...,\alpha_q\}\nonumber\\
&\mathfrak{su}(p,s) ,\,p\leq s <q,  \\
&\{\alpha_{q+1},\,\alpha_{q+2},...,\alpha_{p+q-1},\alpha_1,\alpha_2,...,\alpha_s\}\nonumber\\
&\mathfrak{su}(s,p),\, 1\leq s<p, \, \\
&\{\alpha_{s},\alpha_{s-1},...,\alpha_{2},\alpha_1,\alpha_{p+q-1},\alpha_{p+q-2},...,\alpha_{q+1}\}\nonumber
\end{align*}
\begin{align*}
&\mathfrak{su}(l,s)+\mathfrak{su}(p-l,q-s),\, 1\leq l\leq s, \, p-l\leq q-s, \,   \\
&\begin{array}{l}  \{\alpha_{p+q-l},\alpha_{p+q-l+1},...,\alpha_{p+q-2},\alpha_1,\alpha_2,...,\alpha_s\}\cup\\
 \{-\alpha_{p+q-l-2},...,-\alpha_{q+1},\gamma,-\alpha_q,...,-\alpha_{s+2}\}
\end{array}\nonumber\\
&\mathfrak{su}(s,l)+\mathfrak{su}(p-l,q-s),\, 1\leq s<l <p, \,\\
&\begin{array}{l}  \{\alpha_{s},\alpha_{s-1},...,\alpha_1,\alpha_{p+q-1},...,\alpha_{p+q-l}\}\cup\\
 \{-\alpha_{p+q-l-2},...,-\alpha_{q+1},\gamma,-\alpha_q,...,-\alpha_{s+2}\}\nonumber
\end{array}
\end{align*}

{\bf Maximal Hermitian regular subalgebras of $\mathfrak{so}^*(2p)$}\\

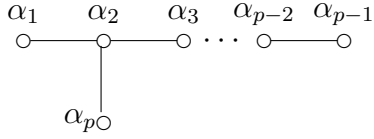
\begin{picture}(50,43)
\multiput(5,31)(30,0){5}{\circle{5}}
\multiputlist(20,31)(30,0)%
{{\line(10,0){25}},{\line(10,0){25}},{$\cdots$},{\line(10,0){25}}}
\multiputlist(5,41)(30,0){$\alpha_{1}$,$\alpha_{2}$,$\alpha_{3}$,$\alpha_{p-2}$,$\alpha_{p-1}$}
\put(35,0){\circle{5}}
\put(34,4){\line(0,10){25}}
\put(20,0){$\alpha_p$}
\end{picture}
\begin{align*}
&\gamma=\alpha_1+2(\alpha_2+...+\alpha_{p-2})+\alpha_{p-1}+\alpha_p \\
&\beta=\alpha_2+ \alpha_3+...+\alpha_p
\end{align*}
\begin{align*}
&\mathfrak{su}(l,p-l) ,\, 1\leq l\leq p/2\, ,\\
&\{-\alpha_{p-l+2},...,-\alpha_{p-2},-\alpha_{p-1},\beta,\alpha_1,...,\alpha_{p-l}\}\nonumber\\
&\mathfrak{so}^*(2l) + \mathfrak{so}^*(2(p-l)) ,\, [p/2]\leq l \leq p-2  \, ,  \\
&\Big\{\begin{array}{l}
\alpha_1,...,\alpha_{l-1}\\
\alpha_p
\end{array}\Big\}
\cup 
\Big\{\begin{array}{l}
\gamma,-\alpha_{p-2},-\alpha_{p-3}...,-\alpha_{l+1}\\
-\alpha_{p-1}
\end{array}\Big\}\nonumber\\
&\mathfrak{so}^*(2(p-1)), \,
\Big\{\begin{array}{l}
 \alpha_1,...,\alpha_{p-2}\\
\alpha_p
\end{array}\Big\}\
\end{align*}

{\bf Maximal Hermitian regular subalgebras of $\mathfrak{so}(p,2)$,  $p= 2k-2$ even} \\

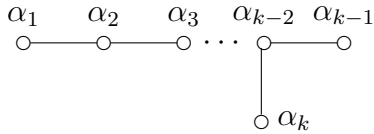
\begin{picture}(50,37)
\multiput(5,31)(30,0){5}{\circle{5}}
\multiputlist(20,31)(30,0)%
{{\line(10,0){25}},{\line(10,0){25}},{$\cdots$},{\line(10,0){25}}}
\multiputlist(5,41)(30,0){$\alpha_{1}$,$\alpha_{2}$,$\alpha_{3}$,$\alpha_{k-2}$,$\alpha_{k-1}$}
\put(94,1){\circle{5}}
\put(94,4){\line(0,10){25}}
\put(100,0){$\alpha_k$}
\end{picture}
\begin{align*}
&\gamma=\alpha_1+2(\alpha_2+...+\alpha_{k-2})+\alpha_{k-1}+\alpha_k\\
&\beta_1=\alpha_2+2(\alpha_3+...+\alpha_{k-2})+\alpha_{k-1}+\alpha_k\\
&\beta_2=\alpha_{k-2}+\alpha_{k-1}+\alpha_k
\end{align*}
\begin{align*}
&\mathfrak{su}(1,1)+ \mathfrak{su}(1,1),\, \{\alpha_1\}\cup\{\gamma\}\\
&\mathfrak{su}(1,k-1), \, \{\alpha_1,...,\alpha_{k-2},\alpha_k\}\\
&\mathfrak{su}(2,2),\, \{\beta_1,\alpha_1,\alpha_2\}\\
&\mathfrak{so}(p-2,2),\, \Big\{\begin{array}{l}
\alpha_1,...,\alpha_{k-2}\\
\beta_2
\end{array}\Big\}
\end{align*}

{\bf Maximal Hermitian regular subalgebras of $\mathfrak{e}_{6(-14)}$}\\

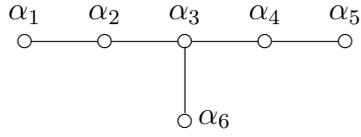
\begin{picture}(50,37)
\multiput(5,31)(30,0){5}{\circle{5}}
\multiputlist(20,31)(30,0)%
{{\line(10,0){25}},{\line(10,0){25}},{\line(10,0){25}},{\line(10,0){25}}}
\multiputlist(5,41)(30,0){$\alpha_{1}$,$\alpha_{2}$,$\alpha_{3}$,$\alpha_{4}$,$\alpha_{5}$}
\put(65,1){\circle{5}}
\put(65,4){\line(0,10){25}}
\put(70,0){$\alpha_6$}
\end{picture}
\begin{align*}
&\gamma=\alpha_1+2\alpha_2+3\alpha_3+2\alpha_4+\alpha_5+2\alpha_6\\
&\beta_1=\alpha_2+2\alpha_3+2\alpha_4+\alpha_5+\alpha_6\\
&\beta_2=\alpha_3+\alpha_4+\alpha_5+\alpha_6
\end{align*}
\begin{align*}
&\mathfrak{su}(1,5)+\mathfrak{su}(1,1), \, \{\alpha_1,...,\alpha_5\}\cup\{\gamma\}\\
&\mathfrak{su}(1,2)+\mathfrak{su}(1,2), \, \{\alpha_1,\alpha_2\}\cup\{\gamma,-\alpha_6\}\\
&\mathfrak{su}(2,4),\, \{\beta_1,\alpha_1,\alpha_2,\alpha_3,\alpha_6\}\\
&\mathfrak{so}^*(10),\, \Big\{\begin{array}{l}
\alpha_1,\alpha_2,\alpha_3,\alpha_4\\
\beta_2
\end{array}\Big\}\\
&\mathfrak{so}(8,2),\, \Big\{\begin{array}{l}
\alpha_1,\alpha_2,\alpha_3,\alpha_4\\
\alpha_6
\end{array}\Big\}
\end{align*}

{\bf Maximal Hermitian regular subalgebras $\mathfrak{g}$ of $\mathfrak{g}^{'}=e_{7(-25)}$}\\

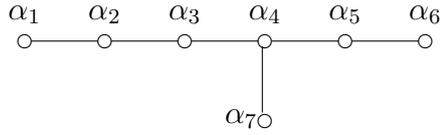
\begin{picture}(50,37)
\multiput(5,31)(30,0){6}{\circle{5}}
\multiputlist(20,31)(30,0)%
{{\line(10,0){25}},{\line(10,0){25}},{\line(10,0){25}},{\line(10,0){25}},{\line(10,0){25}}}
\multiputlist(5,41)(30,0){$\alpha_{1}$,$\alpha_{2}$,$\alpha_{3}$,$\alpha_{4}$,$\alpha_{5}$,$\alpha_{6}$}
\put(95,1){\circle{5}}
\put(94,4){\line(0,10){25}}
\put(80,0){$\alpha_7$}
\end{picture}
\\
\\
\begin{align*}
&\gamma=\alpha_1 + 2\alpha_2 + 3\alpha_3 + 4\alpha_4 + 3\alpha_5 + 2\alpha_6 + 2\alpha_7\\
&\beta_1=\alpha_2 +2\alpha_3 +3\alpha_4+ 2\alpha_5 +\alpha_6 +2\alpha_7\\
&\beta_2=\alpha_3 + 2\alpha_4 + 2\alpha_5 + \alpha_6 + \alpha_7\\
&\beta_3=\alpha_4 + \alpha_5 + \alpha_6 + \alpha_7  
\end{align*}
\begin{align*}
&su(1,5)+su(1,2), \, \{\alpha_1,...,\alpha_4,\alpha_7\}\cup\{\gamma,-\alpha_6\}\\
&su(1,3)+su(1,3), \, \{\alpha_1,\alpha_2,\alpha_3\}\cup\{\gamma,-\alpha_6\,-\alpha_5\}\\
&su(2,6), \, \{\beta_1,\alpha_1,\alpha_2,...,\alpha_6\}\\
&su(3,3), \, \{-\alpha_7,\beta_1,\alpha_1,\alpha_2,\alpha_3\}\\
&so^*(12),\, \Big\{\begin{array}{l}
\alpha_1,\alpha_2,\alpha_3,\alpha_4,\alpha_7\\
\beta_2
\end{array}\Big\}\\
&so(10,2)+su(1,1), \,
\Big\{\begin{array}{l}
\alpha_1,\alpha_2,\alpha_3,\alpha_4,\alpha_5\\
\alpha_7
\end{array}\Big\}
\cup\{\gamma\}\\
&e_{6(-14)},\, \Big\{\begin{array}{l}
\alpha_1,\alpha_3,\alpha_4,\alpha_5,\alpha_5\\
\beta_3
 \end{array}\Big\}
\end{align*}

\renewcommand{\bibname}{{\sc References}}
\begingroup
\let\chapter\section

\endgroup

\ifx\thesis\undefined

\begin{bibdiv}
\begin{biblist}

\bib{C10}{article}{
   author={Burger, Marc},
   author={Iozzi, Alessandra},
   author={Wienhard, Anna},
   title={Surface group representations with maximal Toledo invariant},
   journal={Ann. of Math. (2)},
   volume={172},
   date={2010},
   number={1},
   pages={517--566},
   issn={0003-486X},
}

\bib{C8}{article}{
   author={Burger, Marc},
   author={Iozzi, Alessandra},
   author={Wienhard, Anna},
   title={Tight homomorphisms and Hermitian symmetric spaces},
   journal={Geom. Funct. Anal.},
   volume={19},
   date={2009},
   number={3},
   pages={678--721},
   issn={1016-443X},
}

\bib{C9}{article}{
   author={Dynkin, E. B.},
   title={Semisimple subalgebras of semisimple Lie algebras},
   journal={Am. Math. Soc. Transl. Ser. II, AMS},
   volume={6},
   date={1957},
   pages={111--243},
}



\bib{C31}{article}{
   author={Hamlet, Oskar},
   title={Tight holomorphic maps, a classification},
   journal={J. Lie Theory},
   volume={23},
   date={2013},
   number={3},
   pages={639--654},
   issn={0949-5932},
}

\bib{C32}{article}{
   author={Hamlet, Oskar},
   title={Tight maps and holomorphicity},
   journal={Transformation Groups},
   status={to appear},
}

\bib{C3}{book}{
   author={Helgason, Sigurdur},
   title={Differential geometry, Lie groups, and symmetric spaces},
   series={Graduate Studies in Mathematics},
   volume={34},
   publisher={American Mathematical Society},
   place={Providence, RI},
   date={2001},
   pages={xxvi+641},
}

\bib{C6}{article}{
   author={Ihara, Shin-Ichiro},
   title={Holomorphic imbeddings of symmetric domains},
   journal={J. Math. Soc. Japan},
   volume={19},
   date={1967},
   pages={261--302},
   issn={0025-5645},
}

\bib{C7}{article}{
   author={Satake, Ichir{\^o}},
   title={Holomorphic imbeddings of symmetric domains into a Siegel space},
   journal={Amer. J. Math.},
   volume={87},
   date={1965},
   pages={425--461},
   issn={0002-9327},
}

\end{biblist}
\end{bibdiv}
\end{document}

\fi